\documentclass[12pt]{amsart}

\usepackage{amssymb,mathrsfs,amsmath,amsthm,color,bm,mathtools,bbm,wasysym,cases,mathdots}
\usepackage[pagebackref]{hyperref}
\usepackage[noadjust]{cite}
\usepackage{enumerate}
\usepackage[centering]{geometry}
\allowdisplaybreaks

\newtheorem{theorem}{Theorem}[section]
\newtheorem*{theorem*}{Theorem}

\newtheorem*{question*}{Question}

\newtheorem*{conjecture*}{Conjecture}
\newtheorem{lemma}[theorem]{Lemma}
\newtheorem{proposition}[theorem]{Proposition}
\newtheorem{corollary}[theorem]{Corollary}
\newtheorem{definition}[theorem]{Definition}
\theoremstyle{remark}
\newtheorem{remark}{Remark}
\numberwithin{equation}{section}

\newcommand{\R}{\mathbb{R}}
\newcommand{\Z}{\mathbb{Z}}
\newcommand{\N}{\mathbb{N}}

\newcommand{\cP}{\mathcal{P}}

\newcommand{\SL}{\mathrm{SL}}
\newcommand{\inj}{\operatorname{inj}}

\DeclareMathOperator{\hdim}{\dim_H}

\newcommand{\SO}{\operatorname{SO}}

\newcommand{\qaq}{\mathrm{\quad and\quad}}

\newcommand{\cha}{\mathbbm{1}}
\newcommand{\lm}{\mathcal L}
\renewcommand{\hm}{\mathcal H}

\newcommand{\hc}{\mathcal H_\infty}

\newcommand{\cf}{\mathcal F}

\newcommand{\bx}{\mathbf{x}}

\newcommand{\bp}{\mathbf{p}}

\newcommand{\bv}{\mathbf{v}}

\title[$\tau$-approximable points on self-similar sets]{Hausdorff dimension of $\tau$-approximable points on self-similar sets in $\R^d$}
\author{Yubin He}

\address{Department of Mathematics, Shantou University, Shantou, Guangdong, 515063, China}

\email{ybhemath@outlook.com}

\author{Lingmin Liao}

\address{School of Mathematics and Statistics, Wuhan University, Wuhan, Hubei 430072, China}

\email{lmliao@whu.edu.cn}



\begin{document}
	\begin{abstract}
		Let $d\geq1$. Let $K\subset\R^d$ be a non-singleton self-similar set
		generated by a finite strongly irreducible iterated function system
		satisfying the open set condition, and let $\delta=\hdim K$. For
		$\tau>1/d$, set
		\[
		W_d(\tau)
		=
		\left\{
		\mathbf{x}\in\R^d:
		|q\mathbf{x}-\mathbf{p}|<q^{-\tau}
		\text{ for infinitely many }(\mathbf{p},q)\in\Z^d\times\N
		\right\}.
		\]
		We prove that there exists $\varepsilon_K>0$ such that, for every
		$1/d<\tau<1/d+\varepsilon_K$,
		\[
		\hm^{s(\tau)}(K\cap W_d(\tau))=\infty,
		\qquad\text{with }
		s(\tau):=\delta+\frac{d+1}{1+\tau}-d,
		\]
		and consequently
		\[
		\hdim(K\cap W_d(\tau))
		=
		\delta+\frac{d+1}{1+\tau}-d.
		\]
		In dimension one, specializing to the middle-third Cantor set,
		this establishes the Bugeaud--Durand conjectural formula
		for $\tau>1$ sufficiently close to $1$.
	\end{abstract}
	\maketitle
\section{Introduction}
Let $d\ge1$, and write $|\cdot|$ for the supremum norm on $\R^d$. For
$\bx\in\R^d$, let
\[
\|\bx\|
:=
\inf_{\bp\in\Z^d}
|\bx-\bp|
\]
denote the distance from $\bx$ to the nearest integer vector. Let
$\psi:\N\to(0,\infty)$ be a decreasing function. We say that
$\bx\in\R^d$ is \emph{$\psi$-approximable} if
\[
\|q\bx\|<\psi(q)
\]
for infinitely many $q\in\N$. We denote the set of $\psi$-approximatle points by $W_d(\psi)$. That is,
\[
\begin{split}
	W_d(\psi)
	:=&
	\left\{
	\bx\in\R^d:
	\|q\bx\|<\psi(q)
	\text{ for infinitely many }q\in\N
	\right\}\\
	=&
	\left\{
	\bx\in\R^d:
	|q\bx-p|<\psi(q)
	\text{ for infinitely many }(\bp,q)\in\Z^d\times \N
	\right\}.
\end{split}
\]
Write $W_d(\tau)$ in place
of $W_d(\psi)$ if $\psi(q)=q^{-\tau}$.

A basic starting point in simultaneous Diophantine approximation is
Dirichlet's theorem, which asserts that for every $\mathbf{x}\in\mathbb{R}^d$, there exist infinitely many
$(\mathbf{p},q)\in\mathbb{Z}^d\times\mathbb{N}$ such that
\[
|q\mathbf{x}-\mathbf{p}|\leq q^{-1/d}.
\]
Consequently,
\[
W_d(1/d)=\mathbb{R}^d,
\]
so $1/d$ is the natural critical exponent.

Dirichlet's theorem provides an approximation result that holds for every point.
The corresponding measure-theoretic question concerning the $d$-dimensional
Lebesgue measure, denoted by $\lm_d$, of the set $W_d(\psi)$ is answered by
Khintchine's theorem \cite{Khintmeasure} (for $d=1$), which states that
\[
\lm_d\bigl([0,1]^d\cap W_d(\psi)\bigr)
=
\begin{cases}
	0,
	&
	\displaystyle\sum_{q=1}^{\infty}\psi(q)^d<\infty,
	\\[2ex]
	1,
	&
	\displaystyle\sum_{q=1}^{\infty}\psi(q)^d=\infty.
\end{cases}
\]
Thus, the convergence or divergence of this simple series completely determines
whether almost every point is $\psi$-approximable. The higher dimensional version of this result was established by Schmidt
\cite{SchmidtmeasureTAMS}. Nevertheless, following the terminology in
the literature, we shall continue to refer to the above zero--one law as
Khintchine's theorem, including in the higher dimensional setting.

For $\tau>1/d$, the
Jarn\'ik--Besicovitch theorem gives
\[
\hdim W_d(\tau)
=
\frac{d+1}{1+\tau}.
\]
In fact, the one-dimensional case of the above result was obtained independently by Jarn\'ik
\cite{Jarnikdimension} and Besicovitch
\cite{BesicovitchJLMSDiophantine}, while the higher dimensional version is due to Jarn\'ik \cite{Jarnikmeasure}.

A natural question is what remains true when the ambient space $\R^d$ is
replaced by a lower-dimensional or fractal subset. In the case of fractals,
this problem goes back to Mahler's question on rational approximation of
points in the middle-third Cantor set. More generally, if
$K\subset\R^d$ is a self-similar set, one would like to understand the size
of
\[
K\cap W_d(\tau),
\]
both with respect to the natural self-similar measure on $K$ and in terms of
Hausdorff dimension and Hausdorff measure.

For the measure-theoretic problem, one asks whether analogues of
Khintchine's theorem continue to hold for self-similar measures.
Substantial progress has been made in recent years. Kleinbock, Lindenstrauss
and Weiss \cite{KleinLindWeissSelect} proved, in particular, that friendly
measures are extremal: if $\mu$ is such a measure, then for $\mu$-almost
every $\bx$ there is no $\tau>1/d$ for which
$\bx\in W_d(\tau)$. Thus, in the sense of measure $\mu$, the critical
exponent $1/d$ cannot be improved at almost every point. Remark that the class of
friendly measures includes many natural self-similar measures. Subsequent
work of Einsiedler, Fishman and Shapira \cite{EinFishShafractalsGAFA},
Simmons and Weiss \cite{SimmonsWeissfractalInvent},
Khalil and Luethi \cite{KhalilLuethifractalInvent}, Yu \cite{Yufractals},
and Datta and Jana \cite{DattaJanafractals} established Khintchine-type
results for certain classes of fractal measures, under various additional
hypotheses. In the strongly irreducible self-similar setting, the
measure-theoretic theory was completed by B\'enard, He and Zhang
\cite{BenHeZhangfractalJAMS,BenHeZhangfractalRd}, who proved a full
Khintchine-type zero--one law.

The corresponding Hausdorff dimension problem is also quite subtle. Already in dimension one, it has a rather rich history.
Levesley, Salp and Velani
\cite{LevSalpVelaniMahlerMathAnn} studied Diophantine approximation on the
middle-third Cantor set $K_{1/3}$ by rational points whose denominators are powers of
$3$. In particular, their results imply that
\[
\hdim(K_{1/3}\cap W_1(\tau))
\geq
\frac{\log 2}{\log 3}\cdot\frac{1}{1+\tau}.
\]
This lower bound reflects the arithmetic structure inherent in the middle-third Cantor set: the relevant rational points are essentially the triadic
endpoints arising at different stages of the construction. Levesley, Salp
and Velani also suggested that the dimension of the intersection might simply
be the product of the dimensions of $K_{1/3}$ and $W_1(\tau)$.

Bugeaud and Durand \cite{BugeaudDurandJEMS} subsequently revisited the
problem and proposed a different conjecture. Namely, for every $\tau>1$,
\[
\hdim(K_{1/3}\cap W_1(\tau))
=
\max\left\{
\frac{\log 2}{\log 3}+\frac{2}{1+\tau}-1,\
\frac{\log 2}{\log 3}\cdot\frac{1}{1+\tau}
\right\}.
\]
Near the critical exponent $1$, the first term is dominant.
This naturally leads to a more general question: for a fractal set
$K\subset\mathbb R^d$, under what conditions does the expected dimension
formula
\begin{equation}\label{eq:conjecture}
	\hdim(K\cap W_d(\tau))
	=
	\hdim K+\hdim W_d(\tau)-d
	=
	\hdim K+\frac{d+1}{1+\tau}-d
\end{equation}
hold?

Several partial results in this direction have since been obtained. In the case of $d=1$, Yu \cite{Yufractals} proved the dimension formula \eqref{eq:conjecture} for sufficiently thick missing-digit sets when $\tau>1$ is sufficiently close to $1$, although his result does not currently cover the middle-third Cantor set. Chen \cite{Chenfractal} established the sharp upper bound predicted by \eqref{eq:conjecture} for general self-similar sets, while obtaining a non-trivial but weaker lower bound.

More recently, in our previous work \cite{HeLiaofractal}, we went beyond
the one-dimensional setting and exhibited certain non-trivial self-similar sets in
higher dimensions for which the dimension formula \eqref{eq:conjecture}
holds. Motivated by the works of Yu \cite{Yufractals} and Chen
\cite{Chenfractal}, we also developed a general framework for
Diophantine approximation on fractal sets in $\R^d$, based on a local
rational counting property (see Definition~\ref{d:local counting property}).
Assuming this property, we obtained the sharp upper bound
\begin{equation}\label{eq:upper}
	\hdim(K\cap W_d(\tau))
	\le
	\delta+\frac{d+1}{1+\tau}-d,
\end{equation}
as well as non-trivial lower bounds for $\tau$ in a suitable range.
Moreover, in \cite{HeLiaofractal} we verified the local counting property for strongly irreducible
self-similar sets satisfying the open set condition, thereby making the
sharp upper bound available for this class of self-similar sets.
In the one-dimensional homogeneous setting, one of the resulting lower
bounds is also sharp for a class of sufficiently thick self-similar sets.
Nevertheless, the local counting condition needed to obtain such lower
bounds is too restrictive to treat some of the most basic examples, most
notably the middle-third Cantor set. This limitation is one of the main
motivations for the present work.

In this paper, we settle this problem beyond dimension one by establishing
the exact Hausdorff dimension formula for $K\cap W_d(\tau)$ when $K$ is a
strongly irreducible self-similar set and $\tau$ is sufficiently close to
$1/d$.

\begin{theorem}\label{t:main}
	Let $d\ge1$, and let $K\subset\R^d$ be a non-singleton self-similar set
	generated by a finite strongly irreducible IFS satisfying the open set
	condition. Let
	\[
	\delta:=\hdim K.
	\]
	Then, there exists $\varepsilon_K>0$ such that, for every
	\[
	\frac1d<\tau<\frac1d+\varepsilon_K,
	\]
	we have
	\[\hm^{s(\tau)}(K\cap W_d(\tau))
	=
	\infty,\quad \text{with}\quad s(\tau)
	:=
	\delta+\frac{d+1}{1+\tau}-d,\]
	and in particular,
	\[\hdim(K\cap W_d(\tau))
	=
	\delta+\frac{d+1}{1+\tau}-d.\]
\end{theorem}

Remark that when $d=1$, the
formula becomes
\[
s(\tau)
=
\delta+\frac{2}{1+\tau}-1,
\]
which is precisely the Bugeaud--Durand conjectural
formula when $\tau>1$ is sufficiently close to 1.
	\subsection{New ingredients}

	We begin by recalling the local counting property introduced in our
	previous work \cite{HeLiaofractal}. Since the present paper is concerned
	only with homogeneous Diophantine approximation, we formulate it directly
	in the homogeneous setting other than the general inhomogenous setting of \cite{HeLiaofractal}.

	For \(Q>1\) and \(\eta>0\), define
	\[
	A_Q(\eta)
	:=
	\left\{
	\bx\in\R^d:
	\|q\bx\|<\eta
	\text{ for some }Q\le q<2Q
	\right\}.
	\]

	\begin{definition}\label{d:local counting property}
		Let \(K\subset\R^d\) be a compact set equipped with a non-atomic
		probability measure \(\mu\), and let
		\[
		\alpha>\frac1d,
		\qquad
		\beta>0.
		\]
		We say that \(\mu\) satisfies the
		\emph{\((\alpha,\beta)\)-local counting property}
		if, for all sufficiently large \(Q\), all
		\[
		Q^{-\alpha}\le\eta\le Q^{-1/d},
		\]
		and every ball \(B\) in \(K\) satisfying
		\[
		|B|\ge Q^{-\beta},
		\]
		one has
		\begin{equation}\label{eq:local counting}
			c^{-1}\mu(B)\,Q\eta^d\le 	\mu\bigl(B\cap A_Q(\eta)\bigr)
		\le c
			\mu(B)\,Q\eta^d,
		\end{equation}
		for some constant $c\ge 1$  independent of \(Q\), \(\eta\) and
		\(B\).
	\end{definition}

	The parameter $\beta$ determines the local scale
	at which rational points are effectively counted. Notice that increasing \(\beta\) makes
	Definition~\ref{d:local counting property} substantially stronger, since
	\eqref{eq:local counting} is then required to hold for a larger collection
	of small balls. On the other hand, a sufficiently large value
	of \(\beta\) is essential for obtaining sharp Hausdorff dimension lower
	bounds by the method of \cite{HeLiaofractal}. For instance, as observed
	there, in order for the arguments in \cite{HeLiaofractal} to yield the conjectured dimension for
	the middle-third Cantor set one would need
	\[
	\beta
	>
	2\left(
	1-\frac{\log 2}{\log 3}
	\right)
	\approx 0.7381.
	\]
	The difficulty is that requiring \eqref{eq:local counting} uniformly for
	\emph{every} ball \(B\) is quite strong, and this uniformity
	severely restricts the range of admissible \(\beta\).

	The main new observation of the present paper is that such a uniform local
	counting statement is not actually needed in the proof. At a given scale \(Q^{-\beta}\), we do not need the local counting estimate
	to hold on every ball. It is enough that it holds on ``most" of
	them, as long as the balls where it fails carry only a very small
	proportion of the total mass. Effective recurrence
	(see Theorem \ref{t:BHZ recurrence}) provides precisely the control needed
	to implement this idea and, more importantly, allows us to take \(\beta\)
	arbitrarily close to \(1+1/d\). This is the crucial point.

	\subsection{Notation}

	For a set $E\subset\R^d$, we write $|E|$ for its diameter. We use the notation
	\[
	a\ll b
	\]
	to mean that $a\leq cb$ for some constant $c>0$, and write $a\asymp b$ if
	both $a\ll b$ and $b\ll a$. A subscript indicates the allowed dependence of
	the implied constant; for example,
	\[
	a\ll_u b
	\]
	means that the implied constant may depend on $u$. We also use $\mathrm{O}(\cdot)$ to denote an error term, with the same
	subscript convention, so that $\mathrm{O}_\theta(\cdot)$ allows the implied
	constant to depend on $\theta$. Although the notation $\mathrm{O}(\cdot)$ is visually similar to the symbol $O$ used for an orthogonal matrix, no ambiguity will arise in the proof.

	\subsection{Organization of the paper}

	In Section~\ref{s:preliminaries}, we recall some basic facts about self-similar sets and measures,
	together with the mass transference principle used later in the proof.
	In Section~\ref{s:branch equidistribution}, we use effective equidistribution to obtain rational counting
	estimates.
	In Section~\ref{s:effective recurrence}, we combine these estimates with effective recurrence to obtain
	sharp counting after removing a small exceptional family. Finally, in
	Section~\ref{s:proof}, we apply the results in Section \ref{s:effective recurrence} to prove Theorem~\ref{t:main}.

\section{Preliminaries}	\label{s:preliminaries}

	In this section, we first recall some standard facts about self-similar sets
	and their natural measures. We then introduce a mass transference principle
	that will be used to establish the lower bound for the Hausdorff measure of
	$K\cap W_d(\tau)$.

	\subsection{Self-similar sets}

	Fix a finite set $\Lambda$. An \emph{iterated function system} (IFS for
	short) on $\R^d$ is a finite collection
	\[
	\cf=\{\phi_i:i\in\Lambda\}
	\]
	of contractive similarities of the form
	\[
	\phi_i(\mathbf{x})=\rho_iO_i\mathbf{x}+\mathbf{b}_i,
	\]
	where $0<\rho_i<1$, $O_i\in\SO(d)$ and $\mathbf{b}_i\in\R^d$.

	It was shown by Hutchinson \cite{Hutchinsoninitial} that there exists a unique
	non-empty compact set $K=K(\cf)\subset\R^d$, called the
	\emph{self-similar set} associated with $\cf$, such that
	\begin{equation}\label{eq:self similar set definition}
		K=\bigcup_{i\in\Lambda}\phi_i(K).
	\end{equation}
	Following \cite{Hutchinsoninitial}, we say that $\cf$ satisfies the
	\emph{open set condition} (OSC for short) if there exists a non-empty bounded
	open set $U\subset\R^d$ such that
	\[
	\bigcup_{i\in\Lambda}\phi_i(U)\subset U
	\]
	and
	\[
	\phi_i(U)\cap\phi_j(U)=\emptyset
	\qquad\text{for all }i\neq j.
	\]

	We say that $\cf$ is \emph{strongly irreducible} if there is no finite union
	of proper affine subspaces of $\R^d$ which is invariant under every
	$\phi_i\in\cf$. Here and throughout the sequel, we assume that $\cf$ is
	strongly irreducible and satisfies the OSC.

	The Hausdorff dimension of $K$ is the unique number $\delta\geq0$ satisfying
	\[
	\sum_{i\in\Lambda}\rho_i^\delta=1.
	\]
	Moreover,
	\[
	0<\hm^\delta(K)<\infty.
	\]
	The natural measure on $K$ is the normalized
	$\delta$-dimensional Hausdorff measure
	\[
	\mu
	:=
	\frac{\hm^\delta|_K}{\hm^\delta(K)}.
	\]
	It is a self-similar probability measure satisfying
	\begin{equation}\label{eq:self-similar measure}
		\mu
		=
		\sum_{i\in\Lambda}
		\rho_i^\delta\,(\phi_i)_*\mu
		=
		\sum_{i\in\Lambda}
		\rho_i^\delta\,(\mu\circ\phi_i^{-1}).
	\end{equation}
	Moreover, $\mu$ is $\delta$-Ahlfors regular: for every $\bx\in K$ and every
	$0<r\le|K|$,
	\begin{equation}\label{eq:Ahlfors regular}
		\mu(B(\bx,r))\asymp r^\delta.
	\end{equation}

	We shall also use the standard absolute decay property of irreducible
	self-similar measures satisfying the OSC; see
	\cite[Theorem 2.3]{KleinLindWeissSelect}. That is, there exist
	$C>0$ and $\alpha_0>0$ such that, for every affine
	hyperplane $L\subset\R^d$, every $\bx\in K$, and every
	$0<\varepsilon\leq r\le|K|$,
	\begin{equation}\label{eq:absolute decay}
		\mu\big(B(\bx,r)\cap L^{(\varepsilon)}\big)
		\leq
		C
		\left(\frac{\varepsilon}{r}\right)^{\alpha_0}
		\mu(B(\bx,r)),
	\end{equation}
	where $L^{(\varepsilon)}$ denotes the $\varepsilon$-neighborhood of $L$.
	By decreasing $\alpha_0$ if necessary, we assume that
	$0<\alpha_0\leq d$.

	Let
	\[
	\Lambda^*:=\bigcup_{k\ge1}\Lambda^k
	\]
	denote the set of finite words over $\Lambda$. For later use, for
	$\omega=(i_1,\ldots,i_k)\in\Lambda^k$ with $k\ge2$, we write
	\begin{equation}\label{eq:omega-}
		\omega^-:=(i_1,\ldots,i_{k-1})
	\end{equation}
	its parent word. Given $\omega=(i_1,\ldots,i_k)\in\Lambda^k$, define
	\begin{equation}\label{eq:branch}
		\phi_\omega
		:=
		\phi_{i_1}\circ\cdots\circ\phi_{i_k},
		\qquad
		\rho_\omega
		:=
		\rho_{i_1}\cdots\rho_{i_k},
		\qquad
		K_\omega
		:=
		\phi_\omega(K).
	\end{equation}
	Then, there exist $O_\omega\in\SO(d)$ and $\mathbf{b}_\omega\in\R^d$
	such that
	\begin{equation}\label{eq:phi_omega}
			\phi_\omega(\mathbf{x})
		=
		\rho_\omega O_\omega\mathbf{x}+\mathbf{b}_\omega.
	\end{equation}
	Moreover,
	\[
	\mu(K_\omega)=\rho_\omega^\delta.
	\]

	For each $\omega\in\Lambda^*$, define the corresponding \emph{branch
	measure} by
	\begin{equation}\label{eq:branch measure}
		\mu_\omega
		:=
		(\phi_\omega)_*\mu.
	\end{equation}
	Then, $\mu_\omega$ is a probability measure supported on $K_\omega$, and,
	for every Borel set $A\subset\R^d$,
	\begin{equation}\label{eq:branch conditional identity}
		\mu(K_\omega\cap A)
		=
		\rho_\omega^\delta\,\mu_\omega(A).
	\end{equation}

	Here and throughout the sequel, $\mu$ denotes the natural self-similar
	measure supported on $K$.

\subsection{Hausdorff measure and content}

	For $s>0$, a set $E\subset\R^d$, and $\eta>0$, define
	\[
	\mathcal H_\eta^s(E)
	:=
	\inf\left\{
	\sum_i |B_i|^s:
	E\subset\bigcup_{i\ge1}B_i,\
	B_i\subset\R^d \text{ are balls with } |B_i|\le\eta
	\right\}.
	\]
	The \emph{$s$-dimensional Hausdorff measure} of $E$ is defined by
	\[
	\hm^s(E)
	:=
	\lim_{\eta\to0^+}\mathcal H_\eta^s(E).
	\]
	When $\eta=\infty$, $\hc^s(E)$ is referred to as the
	\emph{$s$-dimensional Hausdorff content} of $E$.

	The following result, known as the mass transference principle from balls to
	open sets, is a fundamental tool for studying the Hausdorff measure of limsup
	sets. It provides a convenient criterion for deriving lower bounds for the
	Hausdorff measure of a limsup set from local Hausdorff content estimates.

	\begin{theorem}[Mass transference principle {\cite[Corollary 2.6]{HeMTPadv}}]
		\label{t:weaken}
		Let $0<s\le\delta$. Suppose that $\{B_i\}$ is a sequence of balls
		centered on $K$, with radii tending to zero, such that
		\[
		\mu\Big(\limsup_{i\to\infty}B_i\Big)=1.
		\]
		Let $\{E_n\}$ be a sequence of open subsets of $\R^d$. Assume that there
		exists a constant $c>0$ such that, for every $i$,
		\begin{equation}\label{eq:condition}
			\limsup_{n\to\infty}
			\hc^s(E_n\cap B_i)
			\ge
			c\,\mu(B_i).
		\end{equation}
		Then,
		\[
		\hm^s
		\Big(
		K\cap\limsup_{n\to\infty}E_n
		\Big)
		=
		\hm^s(K).
		\]
	\end{theorem}

	Accordingly, the main task in the proof of Theorem~\ref{t:main} will be to
	show that certain rational approximation sets satisfy a lower bound of the
	form
	\[
	\hc^s(E_n\cap B)
	\gg
	\mu(B)
	\]
	at the critical exponent $s=s(\tau)$. To establish such a bound, we construct a
	probability measure supported on $E_n\cap B$ and prove a suitable
	Frostman-type estimate. The required lower bound for the Hausdorff content
	then follows from the standard mass distribution principle.

	\begin{proposition}[Mass distribution principle
		{\cite[Lemma 1.2.8]{BishopPeresbook}}]
		\label{p:MDP}
		Let $E\subset\R^d$ be a Borel set supporting a Borel probability measure
		$\nu$. If there exists $C>0$ such that
		\[
		\nu(B(x,r))
		\le
		Cr^s
		\]
		for every $x\in\R$ and every $r>0$, then
		\[
		\hc^s(E)\ge C^{-1}.
		\]
	\end{proposition}

	\section{Effective equidistribution and rational counting on branches}
\label{s:branch equidistribution}

	We first recall the effective equidistribution result for self-similar
	measures in the form needed here, and then apply it to obtain a rational
	counting estimate on each self-similar branch.

	Let
	\[
	G=\SL_{d+1}(\R),
	\qquad
	\Gamma=\SL_{d+1}(\Z),
	\qquad
	X=G/\Gamma,
	\]
	and write
	\[
	\Delta_0:=\Gamma\in X
	\]
	for the standard lattice. Denote by $m_X$ the $G$-invariant
	probability measure on $X$.

	For $t>0$ and $\mathbf{x}\in\R^d$, set
	\begin{equation}\label{eq:a and u}
		a(t)
		=
		\begin{pmatrix}
			t^{1/(d+1)}I_d & 0\\
			0 & t^{-d/(d+1)}
		\end{pmatrix},
		\qquad
		u(\mathbf{x})
		=
		\begin{pmatrix}
			I_d & \mathbf{x}\\
			0 & 1
		\end{pmatrix},
	\end{equation}
	where $I_d$ is the $d\times d$ identity matrix.

	Fix a right-invariant Riemannian metric on $G$, and fix a sufficiently
	small constant $r_0>0$. For $\Delta\in X$, define
	\[
	\inj(\Delta)
	:=
	\sup\left\{
	0<r\le r_0:
	g\mapsto g\Delta
	\text{ is injective on }B_G(e,r)
	\right\},
	\]
	where $B_G(e,r)$ denotes the ball in $G$ centered at the identity element $e$ with radius $r$.
	Note that $\inj(\Delta)$ is small when $\Delta$ lies high in the cusp.

	Let
	\[
	\ell
	:=
	\left\lceil\frac12\dim\SO(d+1)\right\rceil
	=
	\left\lceil\frac{d(d+1)}4\right\rceil,
	\]
	and let $\mathcal S_{\infty,\ell}$ denote a fixed Sobolev norm of order
	$\ell$ on $C^\infty(X)$.

	We will use the following effective equidistribution theorem of
	B\'enard, He and Zhang.

	\begin{theorem}[{\cite[Theorem 1.2]{BenHeZhangfractalRd}}]
		\label{t:BHZ effective equidistribution}
		There exists $\kappa>0$ such that, for every $t\ge1$, every
		$\Delta\in X$, and every smooth function $f$ with
		$\mathcal S_{\infty,\ell}(f)<\infty$, we have
		\begin{equation}\label{eq:BHZ effective equidistribution}
			\left|
			\int_{\R^d}
			f(a(t)u(\mathbf{x})\Delta)\,d\mu(\mathbf{x})
			-
			\int_X f\,dm_X
			\right|
			\ll
			\inj(\Delta)^{-1}
			\mathcal S_{\infty,\ell}(f)
			t^{-\kappa},
		\end{equation}
		where the implied constant depends only on the underlying IFS.
	\end{theorem}

	A useful feature of \eqref{eq:BHZ effective equidistribution} is that the
	dependence on $\Delta$ is given explicitly by its injectivity radius.
	We next apply this estimate to each self-similar branch.

	For a word $\omega$, write
	\[
	k_\omega
	:=
	\begin{pmatrix}
		O_\omega&0\\
		0&1
	\end{pmatrix},
	\]
	where $O_\omega$ is given as in \eqref{eq:phi_omega}. Define
	\begin{equation}\label{eq:h omega}
		h_\omega
		:=
		k_\omega^{-1}a(\rho_\omega^{-1})u(\mathbf{b}_\omega).
	\end{equation}
	More generally, for
	\[
	\phi(\mathbf{x})=rO\mathbf{x}+\mathbf{b},
	\]
	we write
	\[
	k_\phi:=
	\begin{pmatrix}
	O&0\\
	0&1
	\end{pmatrix},
	\qquad
	h_\phi
	:=
	k_\phi^{-1}a(r^{-1})u(\mathbf{b}).
	\]
	The correspondence \(\phi\mapsto h_\phi\) is an anti-homomorphism, so the
	composition reverses the order of multiplication. That is, if
	\[
	\omega=(i_1,\ldots,i_n),
	\qquad
	\phi_\omega
	=
	\phi_{i_1}\circ\cdots\circ\phi_{i_n},
	\]
	then
	\begin{equation}\label{eq:antihomomorphism}
		h_\omega
		=
		h_{\phi_\omega}
		=
		h_{i_n}\cdots h_{i_1}.
	\end{equation}
	Moreover, by the definition of \(h_\omega\), we have
	\[
	\begin{aligned}
		k_\omega a(t\rho_\omega)u(\mathbf{x})h_\omega
		&=
		k_\omega a(t\rho_\omega)u(\mathbf{x})
		k_\omega^{-1}
		a(\rho_\omega^{-1})u(\mathbf{b}_\omega)
		\\
		&=
		a(t\rho_\omega)
		u(O_\omega\mathbf{x})
		a(\rho_\omega^{-1})u(\mathbf{b}_\omega)
		\\
		&=
		a(t)
		u(\rho_\omega O_\omega\mathbf{x})
		u(\mathbf{b}_\omega)
		\\
		&=
		a(t)
		u(\rho_\omega O_\omega\mathbf{x}+\mathbf{b}_\omega)
		\\
		&=
		a(t)u(\phi_\omega(\mathbf{x})).
	\end{aligned}
	\]
	Since $\mu_\omega=(\phi_\omega)_*\mu$, it follows that
	\[
	\int_{\R^d}
	f(a(t)u(\mathbf{x})\Delta)\,d\mu_\omega(\mathbf{x})
	=
	\int_{\R^d}
	f\big(
	k_\omega a(t\rho_\omega)u(\mathbf{x})h_\omega\Delta
	\big)\,d\mu(\mathbf{x}).
	\]
	The matrices $k_\omega$ lie in a fixed compact subgroup of $G$.
	Therefore Theorem~\ref{t:BHZ effective equidistribution}, applied to the
	test function $f(k_\omega\,\cdot)$, gives the following estimate.

	\begin{corollary}
		\label{c:retained injectivity equidistribution}
		Let $\omega\in\Lambda^*$, let $t\ge \rho_\omega^{-1}$, and let
		$\Delta\in X$. Then, for every smooth function $f$ satisfying
		$\mathcal S_{\infty,\ell}(f)<\infty$, we have
		\begin{align}
			\left|
			\int_{\R^d}
			f(a(t)u(\mathbf{x})\Delta)\,d\mu_\omega(\mathbf{x})
			-
			\int_X f\,dm_X
			\right|
			\ll
			\inj(h_\omega\Delta)^{-1}
			\mathcal S_{\infty,\ell}(f)
			(t\rho_\omega)^{-\kappa}.
			\label{eq:retained injectivity equidistribution}
		\end{align}
		The exponent $\kappa$ is the same as in
		Theorem~\ref{t:BHZ effective equidistribution}, and the implied
		constant depends only on the underlying IFS.
	\end{corollary}

	\begin{remark}
		The usual uniform estimate is obtained by further bounding
		$\inj(h_\omega\Delta)^{-1}$ in terms of $\rho_\omega^{-1}$.
		We do not make this replacement. Instead, we keep the actual
		branch dependent quantity $\inj(h_\omega\Delta)^{-1}$ and use effective
		recurrence in the next section to show that it is sufficiently small on
		all but a small family of branches.
	\end{remark}

	We next turn this branchwise equidistribution estimate into rational
	counting. Recall that for $Q\ge1$ and $\eta>0$,
	\begin{equation}\label{eq:primitive shell}
		A_Q(\eta)
		=
		\left\{
		\mathbf{x}\in\R^d:
		\begin{array}{l}
			\text{there exist }(\mathbf{p},q)\in\Z^d\times\N\text{ such that}\\
			Q<q\le2Q,\qquad
			|q\mathbf{x}-\mathbf{p}|<\eta
		\end{array}
		\right\}.
	\end{equation}
	We shall apply a standard smooth approximation to the corresponding cusp
	target. For completeness, let
	\[
	\mathcal E_Q(\eta)
	:=
	\left\{
	\mathbf{x}\in\R^d:
	\begin{array}{l}
		\text{there exists a primitive }
		(\mathbf{p},q)\in\Z^d\times\N\text{ such that}\\
		0<q\le Q,\qquad
		|q\mathbf{x}-\mathbf{p}|<\eta
	\end{array}
	\right\}.
	\]
	Set
	\begin{equation}\label{eq:high dimensional R T}
		R:=(Q\eta^d)^{\frac{1}{d+1}},
		\qquad
		T:=\frac{Q}{\eta},
	\end{equation}
	and define
	\[
	\mathscr R_{Q,\eta}:
	=
	(-R,R)^d\times(0,R].
	\]
	Then, the Dani correspondence gives
	\begin{equation}\label{eq:high dimensional Dani correspondence}
		\bx\in\mathcal E_Q(\eta)
		\quad\Longleftrightarrow\quad
		a(T)u(\bx)\Delta_0
		\text{ contains a primitive vector in }
		\mathscr R_{Q,\eta}.
	\end{equation}
	Indeed,
	\[
	a(T)u(\bx)
	\binom{-\bp}{q}
	=
	\binom{
		T^{1/(d+1)}(q\bx-\bp)
	}{
		T^{-d/(d+1)}q
	}.
	\]
	By \eqref{eq:high dimensional R T},
	\[
	T^{1/(d+1)}\eta
	=
	T^{-d/(d+1)}Q
	=
	R.
	\]
	Hence,
	\[
	|q\bx-\bp|<\eta,
	\qquad
	0<q\le Q
	\]
	is equivalent to
	\[
	a(T)u(\bx)
	\binom{-\bp}{q}
	\in
	\mathscr R_{Q,\eta}.
	\]

	Let
	\[
	\cP(\Z^{d+1})
	:=
	\left\{
	(\bp,q)\in\Z^d\times\Z:
	\gcd(p_1,\ldots,p_d,q)=1
	\right\}.
	\]
	For a compactly supported function
	\(F:\R^{d+1}\to\R\), define its primitive Siegel transform by
	\[
	\widetilde F(g\Delta_0)
	=
	\sum_{\bv\in\cP(\Z^{d+1})}
	F(g\bv).
	\]
	Then, \eqref{eq:high dimensional Dani correspondence} is equivalently the
	condition
	\[
	\widetilde{\cha}_{\mathscr R_{Q,\eta}}
	\bigl(a(T)u(\bx)\Delta_0\bigr)>0.
	\]
	Adapting the smoothing argument of Khalil--Luethi
	\cite[Proposition 8.4 and the proofs of Theorem 9.1 and Lemma 12.7]
	{KhalilLuethifractalInvent},
	we obtain the following.

	\begin{lemma}
		\label{l:high dimensional smoothing}
		For every fixed sufficiently small $\theta>0$, every sufficiently large
		$Q$, and every
		\[
		0<\eta\le Q^{-1/d},
		\]
		there are nonnegative smooth functions
		\[
		\Phi^-_{Q,\eta},\Phi^+_{Q,\eta}\in C^\infty(X),
		\]
		invariant under the left action of
		$\operatorname{diag}(\SO(d),1)$, such that, with $T=Q/\eta$,
		\begin{equation}\label{eq:smoothing sandwich}
			\Phi^-_{Q,\eta}
			(a(T)u(\mathbf{x})\Delta_0)
			\le
			\cha_{\mathcal E_Q(\eta)}(\mathbf{x})
			\le
			\Phi^+_{Q,\eta}
			(a(T)u(\mathbf{x})\Delta_0),
		\end{equation}
		\begin{equation}\label{eq:smoothing norm}
			\mathcal S_{\infty,\ell}(\Phi^\pm_{Q,\eta})
			\ll_\theta 1,
		\end{equation}
		and, for some constant $c_d>0$,
		\begin{align}
			\int_X\Phi^-_{Q,\eta}\,dm_X
			&\ge
			(1-\theta)c_dQ\eta^d
			-
			\mathrm{O}_\theta\bigl((Q\eta^d)^2\bigr),
			\label{eq:smoothing lower mean}\\
			\int_X\Phi^+_{Q,\eta}\,dm_X
			&\le
			(1+\theta)c_dQ\eta^d,
			\label{eq:smoothing upper mean}
		\end{align}
		where the implied constant in $\mathrm{O}_\theta(\cdot)$ is independent of $Q$ and $\eta$.
	\end{lemma}

	Combining this lemma with the branchwise equidistribution estimate,
	Corollary~\ref{c:retained injectivity equidistribution}, gives the rational
	counting statement needed later.

	\begin{proposition}
		\label{p:retained injectivity counting}
		There exists $c_0>0$ such that the following holds. Let $Q$ be
		sufficiently large and suppose that
		\[
		0<\eta\le c_0Q^{-1/d}.
		\]
		Let $\omega\in\Lambda^*$ satisfy
		\[
		T\rho_\omega\ge1
		\qquad\text{with}\ \
		T=\frac{Q}{\eta}.
		\]
		Then,
		\begin{equation}\label{eq:retained lower count}
			\mu_\omega(A_Q(\eta))
			\gg
			Q\eta^d
			-
			\mathrm{O}\left(
			\inj(h_\omega\Delta_0)^{-1}
			(T\rho_\omega)^{-\kappa}
			\right),
		\end{equation}
		and
		\begin{equation}\label{eq:retained upper count}
			\mu_\omega(A_Q(\eta))
			\ll
			Q\eta^d
			+
			\mathrm{O}\left(
			\inj(h_\omega\Delta_0)^{-1}
			(T\rho_\omega)^{-\kappa}
			\right),
		\end{equation}
		where $\kappa>0$ is the exponent appearing in
		Theorem~\ref{t:BHZ effective equidistribution}, and the implied constants in $\mathrm{O}(\cdot)$
		depend only on the underlying IFS.
	\end{proposition}

	\begin{proof}
		By \eqref{eq:smoothing sandwich},
		\[
		\mu_\omega(\mathcal E_Q(\eta))
		\ge
		\int_{\R^d}
		\Phi^-_{Q,\eta}
		(a(T)u(\mathbf{x})\Delta_0)
		\,d\mu_\omega(\mathbf{x}).
		\]
		Since $T\rho_\omega\ge1$, Corollary~
		\ref{c:retained injectivity equidistribution} gives
		\[
		\begin{aligned}
			\mu_\omega(\mathcal E_Q(\eta))
			\ge{}\;&
			\int_X\Phi^-_{Q,\eta}\,dm_X-
			\mathrm{O}\left(
			\inj(h_\omega\Delta_0)^{-1}
			\mathcal S_{\infty,\ell}(\Phi^-_{Q,\eta})
			(T\rho_\omega)^{-\kappa}
			\right).
		\end{aligned}
		\]
		Using \eqref{eq:smoothing norm} and
		\eqref{eq:smoothing lower mean}, we obtain
		\begin{equation}\label{eq:cumulative lower count}
			\mu_\omega(\mathcal E_Q(\eta))
			\ge
			(1-\theta)c_dQ\eta^d
			-
			\mathrm{O}_\theta\bigl((Q\eta^d)^2\bigr)
			-
			\mathrm{O}_\theta\left(
			\inj(h_\omega\Delta_0)^{-1}
			(T\rho_\omega)^{-\kappa}
			\right).
		\end{equation}
		Similarly,
		\begin{equation}\label{eq:cumulative upper count}
			\mu_\omega(\mathcal E_Q(\eta))
			\le
			(1+\theta)c_dQ\eta^d
			+
			\mathrm{O}_\theta\left(
			\inj(h_\omega\Delta_0)^{-1}
			(T\rho_\omega)^{-\kappa}
			\right).
		\end{equation}

		Note that
		\[
		\mathcal E_{2Q}(\eta)\setminus\mathcal E_Q(\eta)
		\subset
		A_Q(\eta)
		\subset
		\mathcal E_{2Q}(\eta).
		\]
		Hence,
		\[
		\mu_\omega(A_Q(\eta))
		\ge
		\mu_\omega(\mathcal E_{2Q}(\eta))
		-
		\mu_\omega(\mathcal E_Q(\eta)).
		\]
		Applying \eqref{eq:cumulative lower count} with $2Q$ and
		\eqref{eq:cumulative upper count} with $Q$, we obtain
		\[
		\begin{aligned}
			\mu_\omega(A_Q(\eta))
			\ge{}\;&
			\bigl(2(1-\theta)-(1+\theta)\bigr)c_dQ\eta^d
			-
			\mathrm{O}_\theta\bigl((Q\eta^d)^2\bigr)
			\\
			&-
			\mathrm{O}_\theta\left(
			\inj(h_\omega\Delta_0)^{-1}
			(T\rho_\omega)^{-\kappa}
			\right).
		\end{aligned}
		\]
			Choose $\theta>0$ sufficiently small so that
		\[
		2(1-\theta)-(1+\theta)>0,
		\]
		and fix this choice of $\theta$ for the remainder of the proof.
		Since
		\[
		Q\eta^d\le c_0^d,
		\]
		choosing $c_0>0$ sufficiently small allows the quadratic error term $\mathrm{O}_\theta((Q\eta^d)^2)$ to be
		absorbed into the main term. This proves
		\eqref{eq:retained lower count}.

		For the upper bound, since
		\[
		A_Q(\eta)\subset\mathcal E_{2Q}(\eta),
		\]
		equation \eqref{eq:cumulative upper count}, applied with $2Q$ in place of
		$Q$, yields
		\[
		\mu_\omega(A_Q(\eta))
		\le
		2(1+\theta)c_dQ\eta^d
		+
		\mathrm{O}\left(
		\inj(h_\omega\Delta_0)^{-1}
		(2T\rho_\omega)^{-\kappa}
		\right).
		\]
		Since $(2T\rho_\omega)^{-\kappa}\ll
		(T\rho_\omega)^{-\kappa}$, this gives
		\eqref{eq:retained upper count}.
	\end{proof}

\section{Sharp counting on good branches and effective recurrence}
\label{s:effective recurrence}

	We first introduce a family of  branches and show that the counting
	estimate from the previous section is sharp whenever
	$h_\omega\Delta_0$ stays sufficiently far away from the cusp. We then apply the
	effective recurrence theorem of B\'enard, He and Zhang \cite{BenHeZhangfractalRd} (see Theorem \ref{t:BHZ recurrence} below) to show that the
	branches for which this fails have small total mass.

	Choose once and for all
	\begin{equation}\label{eq:choice beta}
		\frac{(d-\alpha_0)(d+1)}{d^2}
		<
		\beta
		<
		1+\frac1d,
	\end{equation}
	where $\alpha_0$ is given in \eqref{eq:absolute decay}.
	Such a choice is possible since $\alpha_0>0$.

	For $Q\ge1$, define
	\begin{equation}\label{eq:stopping antichain}
		\mathcal S_Q(\beta)
		:=
		\left\{
		\omega\in\Lambda^*:
		\rho_\omega\le Q^{-\beta}
		<
		\rho_{\omega^-}
		\right\}.
	\end{equation}
	Then, every $\omega\in\mathcal S_Q(\beta)$ satisfies
	\begin{equation}\label{eq:stopping branch scale}
		\rho_\omega\asymp Q^{-\beta}.
	\end{equation}
	Moreover, since the set of contraction ratios is finite, there exist
	constants $c_\beta^-,c_\beta^+>0$ such that
	\begin{equation}\label{eq:stopping word length}
		c_\beta^-\log Q
		\le
		|\omega|
		\le
		c_\beta^+\log Q
	\end{equation}
	for every $\omega\in\mathcal S_Q(\beta)$ and all sufficiently large $Q$.

	For any $u\in\Lambda^*$, write
	\begin{equation}\label{eq:conditional stopping antichain}
		\mathcal S_Q^u(\beta)
		:=
		\left\{
		uv\in\mathcal S_Q(\beta):
		v\in\Lambda^*
		\right\}.
	\end{equation}
	Then, $\mathcal S_Q^u(\beta)$ consists precisely of those words in
	$\mathcal S_Q(\beta)$ having $u$ as a prefix.

	Let $\kappa>0$ be the exponent appearing in
	Theorem~\ref{t:BHZ effective equidistribution}. Choose
	\begin{equation}\label{eq:choice xi}
		0<\xi<
		\kappa\left(1+\frac1d-\beta\right).
	\end{equation}
	We call a word $\omega\in\mathcal S_Q(\beta)$ \emph{good} if
	\begin{equation}\label{eq:good branch}
		\inj(h_\omega\Delta_0)\ge Q^{-\xi},
	\end{equation}
	and \emph{bad} otherwise.

	\subsection{Sharp counting on good branches}

	\begin{proposition}
		\label{p:good branch counting}
		There exists $\varepsilon>0$ such that, for every
		\[
		\frac1d<\tau<\frac1d+\varepsilon,
		\]
		every sufficiently large $Q$, and every good word
		$\omega\in\mathcal S_Q(\beta)$, we have
		\begin{equation}\label{eq:good branch counting}
			\mu_\omega(A_Q(c_0Q^{-\tau}))
			\asymp
			Q^{1-d\tau},
		\end{equation}
		where $c_0>0$ is chosen sufficiently small as in
		Proposition~\ref{p:retained injectivity counting}, and the implied
		constants are independent of $Q$ and $\omega$.
	\end{proposition}

	\begin{proof}
		Set $\eta:=c_0Q^{-\tau}$. For
		$\omega\in\mathcal S_Q(\beta)$, we have
		\[
		T:=\frac{Q}{\eta}\asymp Q^{1+\tau}\qaq
		T\rho_\omega
		\asymp
		Q^{1+\tau-\beta}.
		\]
		Since $\beta<1+1/d$ and $\tau>1/d$, we have
		$T\rho_\omega\ge1$ for all sufficiently large $Q$.

		If $\omega$ is good, then
		\[
		\inj(h_\omega\Delta_0)^{-1}\le Q^\xi.
		\]
		Proposition~\ref{p:retained injectivity counting} therefore gives
		upper and lower bounds with main term $Q^{1-d\tau}$ and error term
		\[
		\mathrm{O}\left(
		Q^{\xi-\kappa(1+\tau-\beta)}
		\right).
		\]
		This error is $\ll Q^{1-d\tau}$ provided that
		\begin{equation}\label{eq:dynamic parameter inequality}
			\xi+(d\tau-1)
			<
			\kappa(1+\tau-\beta).
		\end{equation}
		At $\tau=1/d$, this reduces to
		\[
		\xi<
		\kappa\left(1+\frac1d-\beta\right),
		\]
		which holds by \eqref{eq:choice xi}. Hence, by choosing
		$\varepsilon>0$ sufficiently small,
		\eqref{eq:dynamic parameter inequality} holds throughout the stated
		range. The error can therefore be absorbed into the main term.
	\end{proof}

	It remains to show that the bad words form only a small exceptional
	family. This is where effective recurrence enters.

	\subsection{Effective recurrence}

		Recall from \eqref{eq:h omega} that $h_i=h_{\phi_i}$ for each
	$i\in\Lambda$. Define a discrete Borel measure on $G$ by
	\begin{equation}\label{eq:random walk measure}
		\mathfrak m
		:=
		\sum_{i\in\Lambda}
		\rho_i^\delta\,\delta_{h_i},
	\end{equation}
	where $\delta_{h_i}$ denotes the Dirac measure at $h_i$.
	Since
	\[
	\sum_{i\in\Lambda}\rho_i^\delta=1,
	\]
	$\mathfrak m$ is a probability measure on $G$. Moreover, by the
	anti-homomorphism property \eqref{eq:antihomomorphism}, its $n$-fold
	convolution power is given by
	\begin{equation}\label{eq:convolution symbolic identity}
		\mathfrak m^{*n}
		=
		\sum_{\omega\in\Lambda^n}
		\rho_\omega^\delta\,\delta_{h_\omega}.
	\end{equation}

	We need the following effective recurrence estimate.

	\begin{theorem}[{\cite[Proposition 4.1]{BenHeZhangfractalRd}}]
		\label{t:BHZ recurrence}
		There exist constants
		\[
		\kappa_1>0,
		\qquad
		\kappa_1'>0
		\]
		such that, for every $y\in X$, every $n\ge1$, and every $0<r<1$,
		\begin{equation}\label{eq:BHZ recurrence}
			\mathfrak m^{*n}*\delta_y
			\big\{z\in X:\inj(z)<r\big\}
			\ll
			\left(
			\inj(y)^{-\kappa_1'}e^{-\kappa_1n}
			+1
			\right)
			r^{\kappa_1}.
		\end{equation}
	\end{theorem}

	The following estimate shows that the bad words have
small total mass inside every fixed $K_\omega$.

	\begin{proposition}
		\label{p:conditional pruning}
		Fix $u\in\Lambda^*$. Then, for every
		\[
		0<\gamma<\kappa_1\xi,
		\]
		and all sufficiently large $Q$,
		\begin{equation}\label{eq:conditional pruning}
			\sum_{\substack{
					\omega\in\mathcal S_Q^u(\beta)\\
					\inj(h_\omega\Delta_0)<Q^{-\xi}
			}}
			\mu(K_\omega)
			\ll_u
			Q^{-\gamma}\mu(K_u).
		\end{equation}
	\end{proposition}

	\begin{proof}
		Every $\omega\in\mathcal S_Q^u(\beta)$ can be written uniquely as
		$\omega=uv$ with $v\in\Lambda^*$. Since
		$\mu(K_\omega)=\rho_\omega^\delta$, we have
		\begin{equation}\label{eq:relative cylinder mass}
			\frac{\mu(K_{uv})}{\mu(K_u)}
			=
			\rho_v^\delta.
		\end{equation}
		Since $u$ is fixed, \eqref{eq:stopping word length} implies that there
		exist constants $c_u^-,c_u^+>0$ such that
		\[
		c_u^-\log Q
		\le
		|v|
		\le
		c_u^+\log Q
		\]
		whenever $uv\in\mathcal S_Q^u(\beta)$.

		By the anti-homomorphism property, $h_{uv}=h_vh_u$. Grouping the
		words $v$ according to their length and using \eqref{eq:relative cylinder mass} and
		\eqref{eq:convolution symbolic identity}, we obtain
		\begin{align}
			&\frac{1}{\mu(K_u)}
			\sum_{\substack{
					\omega\in\mathcal S_Q^u(\beta)\\
					\inj(h_\omega\Delta_0)<Q^{-\xi}
			}}
			\mu(K_\omega)
			\nonumber\\
			\le{}&
			\sum_{n=\lceil c_u^-\log Q\rceil}^{\lfloor c_u^+\log Q\rfloor}
			\mathfrak m^{*n}*\delta_{h_u\Delta_0}
			\left\{
			z\in X:
			\inj(z)<Q^{-\xi}
			\right\}.
			\label{eq:bad branch mass bound}
		\end{align}
		Theorem~\ref{t:BHZ recurrence} gives
		\[
		\mathfrak m^{*n}*\delta_{h_u\Delta_0}
		\left\{
		z\in X:
		\inj(z)<Q^{-\xi}
		\right\}
		\ll_u
		Q^{-\kappa_1\xi}.
		\]
		Since there are $\ll_u\log Q$ possible values of $n$,
		\[
		\frac{1}{\mu(K_u)}
		\sum_{\substack{
				\omega\in\mathcal S_Q^u(\beta)\\
				\inj(h_\omega\Delta_0)<Q^{-\xi}
		}}
		\mu(K_\omega)
		\ll_u
		(\log Q)Q^{-\kappa_1\xi}.
		\]
		For every $0<\gamma<\kappa_1\xi$, the right-hand side is
		$\ll_uQ^{-\gamma}$, proving \eqref{eq:conditional pruning}.
	\end{proof}

\section{Proof of Theorem~\ref{t:main}}
\label{s:proof}

	We now turn the counting estimates obtained above into a lower bound for
	the Hausdorff measure of $K\cap W_d(\tau)$.  The argument follows closely that of our previous work \cite{HeLiaofractal}, with the main difference being that the required counting input is now provided by the estimates established above.

	Let $\alpha_0$, $\beta$, $\xi$, and $\varepsilon>0$ be as in \eqref{eq:absolute decay},
	\eqref{eq:choice beta}, \eqref{eq:choice xi}, and
	Proposition~\ref{p:good branch counting}, respectively. For
	\[
	\frac1d<\tau<\frac1d+\varepsilon,
	\]
	set
	\begin{equation}\label{eq:s tau}
		s=s(\tau)
		:=
		\delta+\frac{d+1}{1+\tau}-d
		=
		\delta-\frac{d\tau-1}{1+\tau}.
	\end{equation}
	Then,
	\begin{equation}\label{eq:delta minus s}
		\delta-s
		=
		\frac{d\tau-1}{1+\tau}
		=
		\frac{d(\tau-1/d)}{1+\tau}.
	\end{equation}

	By decreasing $\varepsilon>0$ if necessary, we may assume that throughout
	this range
	\begin{equation}\label{eq:geometric parameter conditions}
		\frac{\beta\delta}{s}
		<
		1+\frac1d,
		\qquad
		\delta-s<\alpha_0,
		\qquad
		\beta>
		\frac{(d-\alpha_0)(1+\tau)s}{d\delta}.
	\end{equation}
	This is possible by \eqref{eq:choice beta}, since $s(\tau)\to\delta$ as
	$\tau\downarrow1/d$.

	We shall prove the following local estimate.

	\begin{proposition}
		\label{p:local Hausdorff content}
		For every ball $B$ centered on $K$ and every
		$1/d<\tau<1/d+\varepsilon$,
		\begin{equation}\label{eq:local Hausdorff content}
			\mathcal H_\infty^{s(\tau)}
			\bigl(B\cap A_Q(c_0Q^{-\tau})\bigr)
			\gg
			\mu(B)
		\end{equation}
		for all sufficiently large $Q$ depending on $B$, where $c_0>0$ is chosen sufficiently small as in Proposition \ref{p:retained injectivity counting} and the implied constant is independent
		of $B$ and $Q$.
	\end{proposition}

	\begin{proof}
		We first construct a suitable subset of
		$B\cap A_Q(c_0Q^{-\tau})$. By self-similarity and the OSC, there exists
		\[
		J=K_u\subset B
		\]
		with $u\in\Lambda^*$ such that
		\begin{equation}\label{eq:cylinder inside ball}
			|J|\asymp |B|
			\qaq
			\mu(J)\asymp\mu(B).
		\end{equation}
		Note that $J$ and the word $u$ depend on $B$.
		Let
		\[
		\mathcal G_Q(J)
		:=
		\left\{
		\omega\in\mathcal S_Q^u(\beta):
		\inj(h_\omega\Delta_0)\ge Q^{-\xi}
		\right\}
		\]
		and set
		\begin{equation}\label{eq:G Q J}
			G_Q(J)
			:=
			\bigcup_{\omega\in\mathcal G_Q(J)}K_\omega.
		\end{equation}
		Finally, define
		\begin{equation}\label{eq:F Q}
			F_Q(J,\tau)
			:=
			G_Q(J)\cap A_Q(c_0Q^{-\tau}).
		\end{equation}

		By Proposition~\ref{p:conditional pruning}, for any fixed
		$0<\gamma<\kappa_1\xi$,
		\[
		\sum_{\omega\in\mathcal G_Q(J)}
		\mu(K_\omega)
		=
		(1+c_uQ^{-\gamma})\mu(J),
		\]
		for some constant $c_u>0$ depending on $u$.
		Hence,
		\begin{equation}\label{eq:good mass in J}
			\sum_{\omega\in\mathcal G_Q(J)}
			\mu(K_\omega)
			\asymp
			\mu(J)
		\end{equation}
		whenever $Q$ is sufficiently large, where the required lower bound on $Q$ depends on $u$, and hence on $B$. On the other hand,
		Proposition~\ref{p:good branch counting} gives, for every
		$\omega\in\mathcal G_Q(J)$,
		\begin{equation}\label{eq:good counting}
			\mu_\omega(A_Q(c_0Q^{-\tau}))
			\asymp
			Q^{1-d\tau}.
		\end{equation}
		Therefore,
		\begin{align}
			\mu(F_Q(J,\tau))
			&=
			\sum_{\omega\in\mathcal G_Q(J)}
			\mu(K_\omega\cap A_Q(c_0Q^{-\tau}))
			\notag\\
			&\asymp
			Q^{1-d\tau}
			\sum_{\omega\in\mathcal G_Q(J)}\mu(K_\omega)
			\asymp
			Q^{1-d\tau}\mu(B).
			\label{eq:mass pruned set}
		\end{align}

		Define the probability measure
		\begin{equation}\label{eq:lambda Q}
			\lambda_Q
			:=
			\frac{\mu|_{F_Q(J,\tau)}}{\mu(F_Q(J,\tau))}.
		\end{equation}
		We claim that
		\begin{equation}\label{eq:Frostman target}
			\lambda_Q(D)
			\ll
			\frac{r^s}{\mu(B)}
		\end{equation}
		for every ball $D=B(\mathbf{x},r)$. We will distinguish five cases.

		\medskip
		\noindent
		\textbf{Case I}: $r\ge Q^{-\beta}$. Since $F_Q(J,\tau)\subset G_Q(J)$, the upper bound in
		Proposition~\ref{p:good branch counting} gives
		\[
		\begin{aligned}
			\mu(D\cap F_Q(J,\tau))
			&\le
			\sum_{\substack{
					\omega\in\mathcal G_Q(J)\\
					K_\omega\cap D\ne\emptyset
			}}
			\mu(K_\omega\cap A_Q(c_0Q^{-\tau}))
			\\
			&\ll
			Q^{1-d\tau}
			\sum_{\substack{
					\omega\in\mathcal S_Q(\beta)\\
					K_\omega\cap D\ne\emptyset
			}}
			\mu(K_\omega).
		\end{aligned}
		\]
		All $K_\omega$ in the last sum have diameter $\asymp Q^{-\beta}\le r$,
		so their union lies in a fixed enlargement of $D$. By the Ahlfors
		regularity of $\mu$, we have
		\[
		\mu(D\cap F_Q(J,\tau))
		\ll
		Q^{1-d\tau}r^\delta.
		\]
		Together with \eqref{eq:mass pruned set}, this gives
		\begin{equation}\label{eq:estimate1}
			\lambda_Q(D)
			\ll
			\frac{r^\delta}{\mu(B)}
			\le
			\frac{r^s}{\mu(B)}.
		\end{equation}

		\medskip
		\noindent
		\textbf{Case II}:
		$Q^{-\beta\delta/s}\le r<Q^{-\beta}$.
		If $D\cap F_Q(J,\tau)=\emptyset$, there is nothing to prove.
		Otherwise choose $z\in D\cap F_Q(J,\tau)$. Then,
		\[
		D\cap F_Q(J,\tau)
		\subset
		B(z,2Q^{-\beta})\cap F_Q(J,\tau).
		\]
		Applying \eqref{eq:estimate1} in Case I with $r=2Q^{-\beta}$, we obtain
		\[
		\lambda_Q(D)
		\ll
		\frac{Q^{-\beta\delta}}{\mu(B)}
		\le
		\frac{r^s}{\mu(B)}.
		\]

		For the remaining scales we use the simplex lemma (see e.g. \cite[Lemma 4]{KrisThornVelanibadAdv}). In the form needed
		here, it states that there exists $c_d>0$ such that all rational points
		$\mathbf{p}/q$ with $1\le q\le2Q$ contained in a ball of radius
		$c_dQ^{-(1+1/d)}$ lie in a common affine hyperplane.

		\medskip
		\noindent
		\textbf{Case III}:
		$c_dQ^{-(1+1/d)}\le r<Q^{-\beta\delta/s}$.
		By the Ahlfors
		regularity of $\mu$, $K\cap D$ can be covered by
		\[
		\mathrm{O}\left(
		r^\delta Q^{(1+1/d)\delta}
		\right)
		\]
		balls $B_i$ of radius
		\[
		\rho:=c_dQ^{-(1+1/d)}/2.
		\]
		Suppose that $B_i$ meets an approximation neighborhood corresponding to
		$\mathbf p/q$, where $Q<q\le 2Q$. Since
		\[
		\left|\mathbf x-\frac{\mathbf p}{q}\right|
		\ll Q^{-(1+\tau)}
		\]
		for $\mathbf x\in A_Q(c_0Q^{-\tau})$, and $\tau>1/d$, the corresponding
		rational center $\mathbf p/q$ lies in a fixed enlargement of $B_i$ of
		radius comparable to $Q^{-(1+1/d)}$. The simplex lemma therefore implies
		that, for each $B_i$, all such rational centers lie in a common affine
		hyperplane $L_i$. It follows that
		\[
		B_i\cap A_Q(c_0Q^{-\tau})
		\subset
		B_i\cap
		 L_i^{(CQ^{-(1+\tau)})}
		\]
		for some fixed constant $C>0$.
		Applying \eqref{eq:absolute decay} with
		$\varepsilon=CQ^{-(1+\tau)}$ and
		$\rho\asymp Q^{-(1+1/d)}$, we obtain
		\[
		\begin{aligned}
			\mu\bigl(B_i\cap A_Q(c_0Q^{-\tau})\bigr)
			&\ll
			\left(
			\frac{Q^{-(1+\tau)}}{Q^{-(1+1/d)}}
			\right)^{\alpha_0}
			\mu(B_i)\\
			&\ll
			Q^{-\alpha_0(\tau-1/d)}
			Q^{-(1+1/d)\delta},
		\end{aligned}
		\]
		where in the last step we used  the Ahlfors
		regularity of $\mu$. Summing over the
		$\mathrm{O}(r^\delta Q^{(1+1/d)\delta})$ balls $B_i$ gives
		\[
		\mu(D\cap A_Q(c_0Q^{-\tau}))
		\ll
		r^\delta
		Q^{-\alpha_0(\tau-1/d)}.
		\]
		Using \eqref{eq:mass pruned set}, we have
		\begin{equation}\label{eq:case III normalized}
			\lambda_Q(D)
			\ll
			\frac{
			r^\delta
			Q^{(d-\alpha_0)(\tau-1/d)}
			}{
			\mu(B)
			}.
		\end{equation}
		Since $r<Q^{-\beta\delta/s}$ and $\delta-s>0$, we have
		\[
		r^{\delta-s}
		\le
		Q^{-\frac{\beta\delta}{s}(\delta-s)}.
		\]
		Therefore, the right-hand side of
		\eqref{eq:case III normalized} is $\mathrm{O}(r^s/\mu(B))$ provided that
	\begin{equation}\label{eq:d-alpha}
		(d-\alpha_0)\left(\tau-\frac1d\right)
		\le
		\frac{\beta\delta}{s}(\delta-s).
	\end{equation}
		Using
		\[
		\delta-s
		=
		\frac{d\tau-1}{1+\tau}
		=
		\frac{d}{1+\tau}
		\left(\tau-\frac1d\right),
		\]
		and cancelling the positive factor
		$\tau-\frac1d$, we conclude that \eqref{eq:d-alpha} is equivalent to
		\[
		\beta
		\ge
		\frac{(d-\alpha_0)(1+\tau)s}{d\delta},
		\]
		as required in \eqref{eq:geometric parameter conditions}.

		\medskip
		\noindent
		\textbf{Case IV}:
		$Q^{-(1+\tau)}\le r<c_dQ^{-(1+1/d)}$.
		The simplex lemma implies that all rational vectors relevant to
		$D\cap A_Q(c_0Q^{-\tau})$ lie in a common affine hyperplane $L$.
		Hence, by \eqref{eq:absolute decay}, we have
		\[
		\mu(D\cap A_Q(c_0Q^{-\tau}))
		\ll
		r^{\delta-\alpha_0}Q^{-\alpha_0(1+\tau)}.
		\]
			Using \eqref{eq:mass pruned set}, we obtain
		\[
		\lambda_Q(D)
		\ll
		\frac{
		r^{\delta-\alpha_0}
		Q^{d\tau-1-\alpha_0(1+\tau)}
		}{
		\mu(B)
		}.
		\]
		Since $\delta-s<\alpha_0$ (see \eqref{eq:geometric parameter conditions}), we have
		\[
		\delta-s-\alpha_0<0.
		\]
		Using $r\ge Q^{-(1+\tau)}$, it follows that
		\[
		r^{\delta-s-\alpha_0}
		\le
		Q^{(1+\tau)(\alpha_0-\delta+s)}.
		\]
		Therefore,
		\[
		\begin{aligned}
			r^{\delta-\alpha_0}
			Q^{d\tau-1-\alpha_0(1+\tau)}
			&=
			r^s
			r^{\delta-s-\alpha_0}
			Q^{d\tau-1-\alpha_0(1+\tau)}
			\\
			&\le
			r^s
			Q^{d\tau-1-(1+\tau)(\delta-s)}
			=
			r^s,
		\end{aligned}
		\]
		where in the last equality we used \eqref{eq:delta minus s}.		Hence,
		\[
		\lambda_Q(D)
		\ll
		\frac{r^s}{\mu(B)}.
		\]

			\medskip
		\noindent
		\textbf{Case V}:
		$0<r<Q^{-(1+\tau)}$.
		Since $F_Q(J,\tau)\subset K$, the Ahlfors regularity of $\mu$ and \eqref{eq:mass pruned set} give
		\[
		\lambda_Q(D)
		\ll
		\frac{r^\delta Q^{d\tau-1}}{\mu(B)}.
		\]
		Since $r< Q^{-(1+\tau)}$, by \eqref{eq:delta minus s} we have
		\[
		r^{\delta-s}
		<
		Q^{-(1+\tau)(\delta-s)}
		=
		Q^{-(d\tau-1)}.
		\]
		Therefore,
		\[
		r^{\delta-s}Q^{d\tau-1}<1,
		\]
		and hence
		\[
		\lambda_Q(D)
		\ll
		\frac{r^s}{\mu(B)}.
		\]

		The first condition in \eqref{eq:geometric parameter conditions}
		ensures that, for all sufficiently large $Q$, the above five cases
		exhaust all $r>0$. We have therefore proved
		\[
		\lambda_Q(D)
		\ll
		\frac{r^{s(\tau)}}{\mu(B)}\qquad\text{for every ball $D$.}
		\]

		Since $\lambda_Q$ is a probability measure supported on
		$F_Q(J,\tau)$, the mass distribution principle (Proposition \ref{p:MDP}) gives
		\[
		\mathcal H_\infty^{s(\tau)}(F_Q(J,\tau))
		\gg
		\mu(B).
		\]
		Finally, noting that
		\[
		F_Q(J,\tau)
		\subset
		B\cap A_Q(c_0Q^{-\tau}),
		\]
		we finish the proof of \eqref{eq:local Hausdorff content}.
	\end{proof}

	We now apply Proposition~\ref{p:local Hausdorff content} to complete the
	proof of Theorem~\ref{t:main}.

	\begin{proof}[Completion of the proof of Theorem~\ref{t:main}]
		Fix
		\[
		\frac1d<\tau<\frac1d+\varepsilon
		\]
		and set
		\[
		s=s(\tau)
		=
		\delta+\frac{d+1}{1+\tau}-d.
		\]
		Applying Proposition~\ref{p:local Hausdorff content} with $Q=2^n$,
		we obtain, for every ball $B$ centered on $K$,
		\[
		\mathcal H_\infty^s
		\bigl(B\cap A_{2^n}(c_0\,2^{-n\tau})\bigr)
		\gg
		\mu(B)
		\]
		for all sufficiently large $n$, with an implied constant independent
		of $B$ and $n$. This is stronger than condition
		\eqref{eq:condition} in Theorem~\ref{t:weaken}. Hence,
		\[
		\mathcal H^s
		\Big(
		K\cap\limsup_{n\to\infty}
		A_{2^n}(c_0\,2^{-n\tau})
		\Big)
		=
		\mathcal H^s(K)
		=
		\infty.
		\]

		If $\mathbf{x}\in A_{2^n}(c_0\,2^{-n\tau})$, then there exist
		$(\mathbf{p},q)\in\Z^d\times\N$ with
		\[
		2^n<q\le2^{n+1}
		\]
		such that
		\[
		|q\mathbf{x}-\mathbf{p}|
		<
		c_0\,2^{-n\tau}.
		\]
		Hence,
		\[
		\limsup_{n\to\infty}
		A_{2^n}(c_0\,2^{-n\tau})
		\subset
		W_d(\tau).
		\]
		Combining this inclusion with the preceding Hausdorff measure estimate, we
		obtain
		\[
		\mathcal H^s(K\cap W_d(\tau))
		=
		\infty.
		\]
		In particular,
		\begin{equation}\label{eq:lower dimension}
			\dim_H(K\cap W_d(\tau))
			\ge
			\delta+\frac{d+1}{1+\tau}-d.
		\end{equation}

		For strongly irreducible self-similar sets, the matching upper bound for $\tau$ in
		a right neighborhood of $1/d$ follows from
		\cite[Corollary 6.5]{HeLiaofractal} (see also \eqref{eq:upper}). After decreasing
		$\varepsilon$ once more if necessary, this proves
		Theorem~\ref{t:main}.
	\end{proof}
\subsection*{Acknowledgements}

	Y. He was supported by the NSFC (No. 12401108) and partially by a grant from
	the Guangdong Provincial Department of Education (2025KCXT\ D013).

	\bibliographystyle{abbrv}
	\bibliography{bibliography}

\begin{thebibliography}{10}

\bibitem{BenHeZhangfractalRd}
T.~B\'enard, W.~He, and H.~Zhang.
\newblock Khintchine dichotomy and {S}chmidt estimates for self-similar
  measures on $\mathbb{R}^d$.
\newblock {\em Preprint arXiv:2508.09076}, 2025.

\bibitem{BenHeZhangfractalJAMS}
T.~B\'enard, W.~He, and H.~Zhang.
\newblock Khintchine dichotomy for self-similar measures.
\newblock {\em J. Amer. Math. Soc.}, 39(3):587--623, 2026.

\bibitem{BesicovitchJLMSDiophantine}
A.~Besicovitch.
\newblock Sets of fractional dimensions ({IV}): On rational approximation to
  real numbers.
\newblock {\em J. London Math. Soc.}, 9(2):126--131, 1934.

\bibitem{BishopPeresbook}
C.~Bishop and Y.~Peres.
\newblock {\em Fractals in probability and analysis}, volume 162 of {\em
  Cambridge Studies in Advanced Mathematics}.
\newblock Cambridge University Press, Cambridge, 2017.

\bibitem{BugeaudDurandJEMS}
Y.~Bugeaud and A.~Durand.
\newblock Metric {D}iophantine approximation on the middle-third {C}antor set.
\newblock {\em J. Eur. Math. Soc. (JEMS)}, 18(6):1233--1272, 2016.

\bibitem{Chenfractal}
S.~Chen.
\newblock The {H}ausdorff dimension of the intersection of $\psi$-well
  approximable numbers and self-similar sets.
\newblock {\em Preprint arXiv:2510.17096}, 2025.

\bibitem{DattaJanafractals}
S.~Datta and S.~Jana.
\newblock On {F}ourier asymptotics and effective equidistribution.
\newblock {\em Preprint arXiv:2407.11961}, 2024.

\bibitem{EinFishShafractalsGAFA}
M.~Einsiedler, L.~Fishman, and U.~Shapira.
\newblock Diophantine approximations on fractals.
\newblock {\em Geom. Funct. Anal.}, 21(1):14--35, 2011.

\bibitem{HeMTPadv}
Y.~He.
\newblock A unified approach to mass transference principle and large
  intersection property.
\newblock {\em Adv. Math.}, 471:Paper No. 110267, 51, 2025.

\bibitem{HeLiaofractal}
Y.~He and L.~Liao.
\newblock Jarn\'ik-type theorem for self-similar sets.
\newblock {\em Preprint arXiv:2602.01307}, 2026.

\bibitem{Hutchinsoninitial}
J.~Hutchinson.
\newblock Fractals and self-similarity.
\newblock {\em Indiana Univ. Math. J.}, 30(5):713--747, 1981.

\bibitem{Jarnikdimension}
V.~Jarn\'ik.
\newblock {D}iophantische {A}pproximationen und {H}ausdorffsches {M}ass.
\newblock {\em Mat. Sb.}, 36:371--382, 1929.

\bibitem{Jarnikmeasure}
V.~Jarn\'ik.
\newblock {\"U}ber die simultanen diophantischen {A}pproximationen.
\newblock {\em Math. Z.}, 33:505--543, 1931.

\bibitem{KhalilLuethifractalInvent}
O.~Khalil and M.~Luethi.
\newblock Random walks, spectral gaps, and {K}hintchine's theorem on fractals.
\newblock {\em Invent. Math.}, 232(2):713--831, 2023.

\bibitem{Khintmeasure}
A.~Khintchine.
\newblock {E}inige {S}{\"a}tze {\"u}ber {K}ettenbr{\"u}che, mit {A}nwendungen
  auf die {T}heorie der diophantischen {A}pproximationen.
\newblock {\em Math. Ann.}, 92(1):115--125, 1924.

\bibitem{KleinLindWeissSelect}
D.~Kleinbock, E.~Lindenstrauss, and B.~Weiss.
\newblock On fractal measures and {D}iophantine approximation.
\newblock {\em Selecta Math. (N.S.)}, 10(4):479--523, 2004.

\bibitem{KrisThornVelanibadAdv}
S.~Kristensen, R.~Thorn, and S.~Velani.
\newblock Diophantine approximation and badly approximable sets.
\newblock {\em Adv. Math.}, 203(1):132--169, 2006.

\bibitem{LevSalpVelaniMahlerMathAnn}
J.~Levesley, C.~Salp, and S.~Velani.
\newblock On a problem of {K}. {M}ahler: {D}iophantine approximation and
  {C}antor sets.
\newblock {\em Math. Ann.}, 338:97--118, 2007.

\bibitem{SchmidtmeasureTAMS}
W.~Schmidt.
\newblock Metrical theorems on fractional parts of sequences.
\newblock {\em Trans. Amer. Math. Soc.}, 110:493--518, 1964.

\bibitem{SimmonsWeissfractalInvent}
D.~Simmons and B.~Weiss.
\newblock Random walks on homogeneous spaces and {D}iophantine approximation on
  fractals.
\newblock {\em Invent. Math.}, 216(2):337--394, 2019.

\bibitem{Yufractals}
H.~Yu.
\newblock Rational points near self-similar sets.
\newblock {\em Preprint arXiv:2101.05910}, 2021.

\end{thebibliography}

\end{document}